\documentclass[12pt,reqno]{amsart}
\usepackage[T1]{fontenc}
\usepackage[utf8]{inputenc}
\usepackage{amssymb,latexsym,amsthm,amsfonts,amscd, enumerate,ragged2e, comment}
\usepackage{amsmath}
\usepackage{aliascnt}
\usepackage{pstricks-add}
\usepackage{graphics,pstricks}
\usepackage{pgf,tikz}
\usepackage{mathrsfs}
\usepackage{fullpage}
\usepackage{microtype,textcomp,csquotes,newpxtext}
\usetikzlibrary{arrows}
\usepackage[dvipsnames]{xcolor}
\definecolor{darkblue}{rgb}{0,0,0.7}
\definecolor{darkred}{rgb}{0.7,0,0}
\usepackage[colorlinks=true, linkcolor=darkred, citecolor=darkblue, urlcolor=blue, pagebackref=true, breaklinks=true]{hyperref}
\usepackage[capitalise]{cleveref}
\usepackage{enumitem}
\usepackage[margin=0.9in]{geometry}

\newtheorem{proposition}{Proposition}[section]

\newaliascnt{lemma}{proposition}
\newtheorem{lemma}[lemma]{Lemma}
\aliascntresetthe{lemma}

\newaliascnt{theorem}{proposition}
\newtheorem{theorem}[theorem]{Theorem}
\aliascntresetthe{theorem}

\newaliascnt{corollary}{proposition}
\newtheorem{corollary}[corollary]{Corollary}
\aliascntresetthe{corollary}

\newaliascnt{question}{proposition}
\newtheorem{question}[question]{Question}
\aliascntresetthe{question}

\newaliascnt{conjecture}{proposition}
\newtheorem{conjecture}[conjecture]{Conjecture}
\aliascntresetthe{conjecture}

\theoremstyle{definition}

\newaliascnt{remark}{proposition}
\newtheorem{remark}[remark]{Remark}
\aliascntresetthe{remark}

\newaliascnt{claim}{proposition}

\aliascntresetthe{claim}

\newaliascnt{observation}{proposition}

\aliascntresetthe{observation}

\newaliascnt{example}{proposition}

\aliascntresetthe{example}

\newaliascnt{definition}{proposition}
\newtheorem{definition}[definition]{Definition}
\aliascntresetthe{definition}

\crefname{proposition}{proposition}{propositions}
\Crefname{proposition}{Proposition}{Propositions}
\crefname{lemma}{lemma}{lemmas}
\Crefname{lemma}{Lemma}{Lemmas}
\crefname{theorem}{theorem}{theorems}
\Crefname{theorem}{Theorem}{Theorems}
\crefname{corollary}{corollary}{corollaries}
\Crefname{corollary}{Corollary}{Corollaries}
\crefname{question}{question}{questions}
\Crefname{question}{Question}{Questions}
\crefname{conjecture}{conjecture}{conjectures}
\Crefname{conjecture}{Conjecture}{Conjectures}
\crefname{remark}{remark}{remarks}
\Crefname{remark}{Remark}{Remarks}
\crefname{claim}{claim}{claims}
\Crefname{claim}{Claim}{Claims}
\crefname{observation}{observation}{observations}
\Crefname{observation}{Observation}{Observations}
\crefname{example}{example}{examples}
\Crefname{example}{Example}{Examples}
\crefname{definition}{definition}{definitions}
\Crefname{definition}{Definition}{Definitions}
\crefname{question}{question}{questions}
\Crefname{question}{Question}{Questions}

\newcommand{\lk}{{\mathrm{lk}}}
\newcommand{\del}{{\mathrm{del}}}

\usepackage{float}
\usepackage{caption}
\usepackage{subcaption}
\usepackage{comment}
\usetikzlibrary{positioning, shapes, fit, arrows}
\usepackage{tkz-graph}
\tikzstyle{vertex}=[circle, draw, inner sep=0pt, minimum size=6pt]

\def\F{\mathcal{F}}

\def\Z{\mathbb{Z}}

\newcommand{\fdel}{{\rm fdel}}

\newcommand{\PF}{{\rm PF}}

\title[Path-free complexes]{On a complete characterization of  path-free complexes associated with complete multipartite graphs}
\thanks{Last updated: \today}

\author{Priyavrat Deshpande}
\address{Chennai Mathematical Institute, India}
\email{pdeshpande@cmi.ac.in}

\author{Shuchita Goyal}
\address{BITS Pilani, India}
\email{shuchita.goyal@pilani.bits-pilani.ac.in}

\author{Rutuja Sawant}
\address{Chennai Mathematical Institute, India}
\email{rutuja@cmi.ac.in}

\thanks{The first and the third authors are partially supported by a grant from the Infosys Foundation. The second author is partially supported by NFSG grant N5/25/1045.}

\begin{document}

\begin{abstract}
Let $G$ be a graph and let $\PF_t(G)$ denote the simplicial complex whose faces are
vertex subsets whose induced subgraphs contain no path on $t$ vertices. 
These complexes encode a forbidden-subgraph condition as a family of allowed vertex subsets.

In this paper we study $t$-path-free complexes of complete multipartite graphs.
Let
\[
G=K_{n_1,\dots,n_m}, \qquad n_1\le\cdots\le n_m.
\]
We first obtain an explicit
structural decomposition of $\PF_t(G)$ as a union of join complexes, together with
an additional lower-dimensional size-truncation term. 
Using this decomposition, we show that for $t\le 2n_{m-1}-2$ the complex $\PF_t(G)$ is not sequentially Cohen-Macaulay,
while for $t\ge 2 n_{m-1}-1$ it is vertex decomposable.

Consequently, we obtain a complete characterization for complete multipartite graphs: $\PF_t(G)$ is vertex decomposable if and only if $t\ge 2n_{m-1}-1$. Equivalently, this is also exactly the range in which $\PF_t(G)$ is shellable and sequentially Cohen-Macaulay.
We further analyze the topology via a Mayer–Vietoris spectral sequence: for complete bipartite graphs we determine the full homotopy type as an explicit wedge of spheres in all cases.
\end{abstract}
\maketitle

\section{Introduction}

Simplicial complexes arising from graphs provide a bridge between
combinatorics, topology, and commutative algebra: one encodes a forbidden
subgraph condition as a family of allowed vertex sets. In this paper we study
the case where the forbidden pattern is a path on $t$ vertices. For a graph
$G$ and a positive integer $t$, the \emph{$t$-path-free complex} $\PF_t(G)$
consists of all subsets $F\subseteq V(G)$ such that the induced subgraph
$G[F]$ contains no copy of $P_t$. 
These complexes arise as special cases of devoid complexes \cite{DT} associated to graphs.
Our focus is on complete multipartite graphs $G=K_{n_1,\dots,n_m}$.

Simplicial complexes built from graph-theoretic constraints have a rich literature. 
The most classic examples are independence complexes $\mathrm{Ind}(G)$, whose faces are independent sets of a graph $G$; these have been extensively studied from both topological and algebraic perspectives \cite{RG, Kak}
In the algebraic direction, these complexes are intimately related to edge ideals $I(G)$ via the Stanley–Reisner correspondence: Cohen–Macaulay properties of $\mathrm{Ind}(G)$ translate directly to depth conditions on 
$k[x_1,\dots,x_n]/I(G)$. 
A systematic study of when independence complexes are vertex decomposable, shellable, or sequentially Cohen–Macaulay was carried out by Woodroofe \cite{RW09, RW} and others \cite{FAVT, VanTuylVillarreal2008}. 

A natural generalization replaces the forbidden edge by a forbidden path. 
The $t$-path ideal $I_t(G)$, generated by monomials corresponding to paths of length $t$ in $G$ has also been studied extensively by the commutative algebra community. 
See for example \cite{Bi, ConcaDeNegri1999, AKK, AH, SM}. 
The Stanley–Reisner complex of the Alexander dual of $I_t(G)$ is precisely $\PF_t(G)$. 
Algebraic properties of path ideals for specific graph families have been studied in the papers cited before. 
However, a complete combinatorial characterization for complete multipartite graphs - the most natural dense graph family - has not previously been given.

The problem is to determine exactly for which $t$ these complexes have strong
recursive and algebraic structure (vertex decomposable, shellable,
sequentially Cohen-Macaulay). This matters because these properties govern
both the topology of the complex and its Stanley-Reisner algebraic behavior, though we focus here on the combinatorial and topological aspects.
Our main result gives a sharp threshold in terms of the second-largest part
size $n_{m-1}$: $\PF_t(G)$ has all three properties precisely when
$t\ge 2n_{m-1}-1$, while for $t\le 2n_{m-1}-2$ it is not sequentially
Cohen-Macaulay.

The classes vertex decomposable, shellable, and sequentially
Cohen-Macaulay are central in this area, because they capture strong
combinatorial structure. 
 
Our strategy is to convert the path-avoidance condition into an explicit decomposition of $\PF_t(K_{n_1,\dots,n_m})$ as a union of join complexes (together with a lower-dimensional size-truncation part), and then analyze this decomposition.

This leads to a complete classification theorem.

\begin{theorem}\label{thm: VDC - CBG}
Let $G=K_{n_1,\dots,n_m}$ with $n_1\le\cdots\le n_m$. For $t\ge 2$, the following statements are equivalent:
\begin{enumerate}
\item $t\ge 2n_{m-1}-1$,
\item $\PF_t(G)$ is vertex decomposable,
\item $\PF_t(G)$ is shellable,
\item $\PF_t(G)$ is sequentially Cohen-Macaulay.
\end{enumerate}
\end{theorem}

In particular, $t\le 2n_{m-1}-2$ is exactly the non-sequentially Cohen-Macaulay range. 
It is perhaps surprising that the largest part plays no role in the threshold. 
The reason is that a long path can always use vertices of the largest part provided enough vertices remain outside it. 
The obstruction is therefore controlled by the second-largest part, which determines the size of the disconnected link used in the proof of non-sequential Cohen–Macaulayness.

We also use the same structural decomposition to study topology: in the bipartite specialization, we give an explicit homotopy-type description as wedges of spheres. 

The paper is organized as follows.
In \Cref{sec: prep} we recall background on graphs and simplicial complexes.
In \Cref{sec: CMG} we develop the multipartite decomposition and prove the
classification theorem above.
In \Cref{sec:CBG-structure-topology} we specialize to complete bipartite
graphs and obtain more explicit decompositions and homotopy consequences.

\section{Background and Preliminary Results}\label{sec: prep}

In this section, we recall some basic notions and results from graph theory and simplicial complexes that will be used throughout the paper.

\subsection{Graphs}

A \emph{graph} $G=(V,E)$ consists of a finite vertex set $V$ and an edge set $E \subseteq \binom{V}{2}$.
We write $V(G)$ and $E(G)$ for the vertex set and edge set of $G$, respectively. The edge between vertices $u$ and $v$ is denoted by $\{u,v\}$, and $u$ is said to be \emph{adjacent} to $v$ if $\{u,v\} \in E(G)$.  

A graph $H=(V(H),E(H))$ is a \emph{subgraph} of a graph $G$ if $V(H)\subseteq V(G)$ and $E(H)\subseteq E(G)$.  
For a subset $W\subseteq V(G)$, the \emph{induced subgraph} of $G$ on $W$, denoted by $G[W]$, is the graph with vertex set $W$ and edge set
\[
\{\{u,v\}\in E(G) : \{u,v\}\subseteq W\}.
\]

\medskip

A \emph{path} on $n$ vertices, denoted by $P_n$, is the graph with vertex set $\{v_1,\dots,v_n\}$ and edge set
\[
\{\{v_i,v_{i+1}\} \mid 1 \le i \le n-1\}.
\]

We define $\lambda(G)$, the \emph{length} of a longest path of $G$ by
\[
\lambda(G) =
\max \{n-1 \mid \text{$G$ contains a subgraph isomorphic to $P_n$}\}.
\]

\subsection{Simplicial Complexes}

We now recall some standard notions related to simplicial complexes.

\medskip

An \emph{(abstract) simplicial complex} $\Delta$ on a finite set $V$ is a collection of subsets of $V$ such that if $F\in \Delta$, then every subset of $F$ also belongs to $\Delta$. The elements of $\Delta$ are called \emph{faces}, and the maximal faces with respect to inclusion are called \emph{facets}. 

The \emph{dimension} of a face $F$ is defined as $\dim F = |F|-1$, where $|F|$ denotes the cardinality of $F$. The \emph{dimension} of $\Delta$ is the maximum dimension among its faces. A singleton $\{v\}\in \Delta$ is called a \emph{vertex} of $\Delta$.

The \emph{empty complex}, denoted by $\{\emptyset\}$, consists only of the empty set and is regarded as the unique $(-1)$-dimensional simplicial complex, while the \emph{void complex}, denoted by $\emptyset$, has no faces (we follow the convention of \cite{kozlov2007}).

A simplicial complex $\Delta$ is called \emph{pure} if all of its facets have the same dimension; otherwise, it is called \emph{non-pure}.

A simplicial complex consisting of a unique facet of dimension $n$ is called an \emph{$n$-simplex}. If $\Delta$ is a simplex on the vertex set $V$, we denote it by $\Delta_V$. 

\medskip

A \emph{subcomplex} $\Delta'$ of $\Delta$ is a simplicial complex such that every face of $\Delta'$ is also a face of $\Delta$. For $W\subseteq V$, the \emph{induced subcomplex} of $\Delta$ on $W$ consists of all faces of $\Delta$ that are contained in $W$.

For an integer $d\ge 0$, the \emph{$d$-skeleton} of $\Delta$ is the subcomplex consisting of all faces of dimension at most $d$. 
The \emph{pure $d$-skeleton}, denoted by $\Delta^{[d]}$, is the subcomplex whose facets are precisely the $d$-dimensional faces of $\Delta$.
Note that the pure $d$-skeleton differs from the standard $d$-skeleton when the complex has facets of dimension less than $d$. 
\medskip

We now recall several fundamental constructions associated with simplicial complexes.

Let $\Delta_1$ and $\Delta_2$ be simplicial complexes on disjoint vertex sets. Their \emph{join}, denoted $\Delta_1 * \Delta_2$, is the simplicial complex with vertex set $V(\Delta_1)\cup V(\Delta_2)$ and face set
\[
\{F_1 \cup F_2 : F_1 \in \Delta_1,\ F_2 \in \Delta_2\}.
\]

Let $F$ be a face of a simplicial complex $\Delta$. The \emph{link} of $F$ in $\Delta$ is
\[
\lk_\Delta(F)=\{F' \in \Delta : F\cap F'=\emptyset,\ F\cup F' \in \Delta\}.
\]

The \emph{deletion} and \emph{face deletion} of $F$ in $\Delta$ are defined by
\[
\del_\Delta(F)=\{F' \in \Delta : F'\cap F=\emptyset\},
\quad
\fdel_\Delta(F)=\{F' \in \Delta : F \not\subseteq F'\}.
\]

If $F=\{v\}$ is a vertex, we simply write $\lk_\Delta(v)$, $\del_\Delta(v)$, and $\fdel_\Delta(v)$.
In this paper, we use only the deletion and the link for vertices.

\medskip

We now recall several important combinatorial and topological properties of simplicial complexes.

\begin{definition}\label{def:VD}
A simplicial complex $\Delta$ is \emph{vertex decomposable} if either $\Delta$ is a simplex (including $\emptyset$ and $\{\emptyset\}$), or there exists a vertex $v$ such that both $\lk_{\Delta}(v)$ and $\del_{\Delta}(v)$ are vertex decomposable, and every facet of $\del_{\Delta}(v)$ is also a facet of $\Delta$. Such a vertex $v$ is called a \emph{shedding vertex}.
\end{definition}

\medskip

Next we recall the notions of shellable and sequentially Cohen-Macaulay simplicial complexes.
We use the recursive definition of shellability, which is equivalent to the classical facet-ordering definition (see \cite[Chapter 3]{JJ08}).

\begin{definition}\label{def:shellable}
A simplicial complex $\Delta$ is \emph{shellable} if either $\Delta$ is a simplex (including $\emptyset$ and $\{\emptyset\}$), or there exists a face $F$ such that both $\lk_{\Delta}(F)$ and $\fdel_{\Delta}(F)$ are shellable, and every facet of $\fdel_{\Delta}(F)$ is also a facet of $\Delta$. Such a face $F$ is called a \emph{shedding face}.
\end{definition}

It is well known (see \cite[Theorem 3.35]{JJ08} and \cite[Theorem 12.3]{kozlov2007}) that if $\Delta$ is a shellable simplicial complex, then $\Delta$ is homotopy equivalent to a wedge of spheres. Moreover, a sphere of dimension $d$ appears in this wedge only if $\Delta$ has facets of dimension $d$.

\begin{definition}\label{def:CM}
A simplicial complex $\Delta$ is \emph{Cohen-Macaulay} over a field $\mathbb{K}$ if
\[
\widetilde{H}_i(\lk_{\Delta}(\sigma);\mathbb{K}) = 0
\]
for all $\sigma \in \Delta$ and all $i < \dim \lk_{\Delta}(\sigma)$.
\end{definition}

\begin{definition}\label{def:SCM}
A simplicial complex $\Delta$ is \emph{sequentially Cohen-Macaulay} over a field $\mathbb{K}$ if its pure $d$-skeleton $\Delta^{[d]}$ is Cohen-Macaulay over $\mathbb{K}$ for every $d\ge0$.
\end{definition}

For simplicial complexes, the following implications are well known (see \cite[Theorem 3.33]{JJ08}):
\begin{equation}\label{eq: implications}
\text{Vertex decomposable}
\Rightarrow
\text{Shellable}
\Rightarrow
\text{Sequentially Cohen-Macaulay}.
\end{equation}

We now state several known results that will be used throughout the paper.

\begin{lemma}[{\cite[Equation 9.12]{TMAB}}]\label{lem: homology for join operation}
    Let $\Delta_1$ and $\Delta_2$ be simplicial complexes with disjoint vertex sets. Assume the homology of one of $\Delta_1$ and $\Delta_2$ is always free. Then we have the following isomorphism in homology:
    \[
\widetilde H_n(\Delta_1 * \Delta_2)
\cong
\bigoplus_{i+j=n-1}
\widetilde H_i(\Delta_1)
\otimes
\widetilde H_j(\Delta_2).
\]
\end{lemma}

\begin{lemma}[{\cite[Theorem 3.30]{JJ08}}]\label{join property}
Let $\Delta_1$ and $\Delta_2$ be simplicial complexes with disjoint vertex sets. Then $\Delta_1 * \Delta_2$ is vertex decomposable if and only if both $\Delta_1$ and $\Delta_2$ are vertex decomposable.
\end{lemma}

The following proposition records a basic hereditary property of vertex decomposable complexes (see \cite{JJ08} for details).

\begin{proposition}\label{prop:V.D. of lk}
The link of every face of a vertex decomposable simplicial complex is vertex decomposable.
\end{proposition}

We also recall a classical result concerning skeleta of simplices.

A classical result states that if $\Delta$ is an $n$-simplex, then its pure
$d$-skeleton is vertex decomposable for each $0\le d\le n$
(see \cite[Lemma 3.10]{RW09}).

\medskip

We now introduce the simplicial complexes that will be the main objects of study in this paper.

\begin{definition}
The \textbf{$t$-path-free complex} $\PF_t(G)$ of a graph $G$ is the simplicial complex on the vertex set $V(G)$ whose faces are subsets $F \subseteq V(G)$ such that the induced subgraph $G[F]$ contains no subgraph isomorphic to a path on $t$ vertices.
\end{definition}

\medskip

The following observation describes the behaviour of the $t$-path-free complex under disjoint unions of graphs.

\begin{proposition}\label{prop:disjoint graphs}
Let $G_1$ and $G_2$ be graphs on disjoint vertex sets. Then
\[
\PF_t(G_1 \sqcup G_2)
=
\PF_t(G_1) * \PF_t(G_2).
\]
Consequently, if both $\PF_t(G_1)$ and $\PF_t(G_2)$ are vertex decomposable, then so is $\PF_t(G_1 \sqcup G_2)$.
\end{proposition}

The next proposition shows that when the parameter $t$ exceeds the maximum possible path length in $G$, the $t$-path-free complex has the simplest possible structure.

\begin{proposition}\label{path free for max path}
Let $G$ be a graph such that $t > \lambda(G) + 1$. Then $\PF_t(G)$ is a $(|V(G)|-1)$-simplex.
\end{proposition}

Finally, we record a useful description of the deletion operation for $t$-path-free complexes, which will play an important role in our inductive arguments.

\begin{lemma}\label{lem:deletion}
Let $v \in V(G)$. Then
\[
\del_{\PF_t(G)}(v) = \PF_t(G \setminus v).
\]
\end{lemma}

\begin{proof}
Let $\sigma \in \del_{\PF_t(G)}(v)$. Then $v \notin \sigma$ and the induced subgraph $G[\sigma]$ does not contain a subgraph isomorphic to a path on $t$ vertices. Since $\sigma$ does not contain $v$, we have
\[
G[\sigma] \cong (G \setminus v)[\sigma].
\]
Thus $(G \setminus v)[\sigma]$ also contains no subgraph isomorphic to a path on $t$ vertices, and hence $\sigma \in \PF_t(G \setminus v)$.

Conversely, let $\sigma \in \PF_t(G \setminus v)$. Then $(G \setminus v)[\sigma]$ contains no subgraph isomorphic to a path on $t$ vertices. Since $v \notin \sigma$, we again have
\[
G[\sigma] \cong (G \setminus v)[\sigma],
\]
which implies that $G[\sigma]$ also contains no such path. Therefore $\sigma \in \del_{\PF_t(G)}(v)$.

Hence $\del_{\PF_t(G)}(v) = \PF_t(G \setminus v)$.
\end{proof}

\section{Structural Decomposition and Consequences }\label{sec: CMG}

In this section, we convert the path-avoidance conditions into explicit join decompositions of $\PF_t(K_{n_1,n_2,\ldots,n_m})$, which we use to deduce vertex decomposability for $t \ge 2n_{m-1}-1$ and non-sequentially Cohen-Macaulay for $t \le 2n_{m-1}-2$. We then apply the same structural description to analyze the topology of these complexes.

\begin{definition}[Complete Multipartite Graph]\label{def: CMG}

Let $m \ge 2$ be an integer and let 
\[
n_1 \le n_2 \le \cdots \le n_m
\]
be positive integers. The \emph{complete multipartite graph}, denoted by 
\[
K_{n_1, n_2, \dots, n_m},
\]
is the graph $G=(V,E)$ defined as follows:

\begin{itemize}
    \item The vertex set $V$ is partitioned into $m$ pairwise disjoint subsets
    \[
    V = V_1 \sqcup V_2 \sqcup \cdots \sqcup V_m,
    \]
    where for each $i=1,\dots,m$,
    \[
    V_i = \{v_1^i, v_2^i, \dots, v_{n_i}^i\}.
    \]

    \item The edge set is
    \[
    E = \big\{ \{v_a^i, v_b^j\} \;\big|\; i \ne j,\; 1 \le a \le n_i,\; 1 \le b \le n_j \big\}.
    \]
\end{itemize}

Thus, every vertex in $V_i$ is adjacent to every vertex in $V_j$ for $i \neq j$, and there are no edges between vertices within the same part $V_i$.
\end{definition}

For convenience, throughout this section, we use the shorthand
\[
S:=\sum_{i=1}^m n_i,
\qquad
R:=S-n_m=\sum_{i=1}^{m-1} n_i.
\]
Thus, $S$ is the total number of vertices, and $R$ is the total size of all parts except the largest one. Given a positive integer $t$, we set $k = \left\lfloor \frac{t}{2} \right\rfloor$. 

\subsection{Longest Paths in Complete Multipartite Graphs}

We first characterise the length of the longest path of complete multipartite graphs, which depends on the difference between the total number of vertices and the size of the largest part.

\begin{lemma}\label{lem: number of vertices in a longest path}
Let $G = K_{n_1,\dots,n_m}$ be a complete multipartite graph defined as above. 
Then the number of vertices in a longest path in $G$ is
\begin{equation}\label{max-path-length in CMG}
\lambda(G) + 1 =
\begin{cases}
S, & \text{if } n_m \le R + 1,\\
2R + 1, & \text{if } n_m > R + 1.
\end{cases}
\end{equation}
\end{lemma}

\begin{proof}
    Since there are no edges inside a part, any path in $G$ cannot contain two consecutive vertices from the same part. Thus, along any path, vertices from a given part must be separated by vertices from other parts.

\medskip

\noindent\textbf{Upper bounds}: Let $P$ be any path in $G$, and let $a_i$ denote the number of vertices of $P$ that lie in $V_i$. Let $j$ be such that $a_j = \max_i a_i$.

Each vertex of $P \cap V_j$, except possibly one, must be separated by a vertex from $P \setminus V_j$. Hence
\[
a_j \le \sum_{i \ne j} a_i + 1.
\]
Summing over all parts gives
\[
|P| = \sum_{i=1}^m a_i \le 2 \sum_{i \ne j} a_i + 1 \le 2R + 1.
\]
Thus, every path has at most $2R+1$ vertices.

Also, trivially, $|P| \le S$. Therefore,
\[
|P| \le \min(S, 2R+1).
\]

\medskip

\noindent\textbf{Case 1: $n_m \le R+1$.}

We show that there exists a path using all $S$ vertices.

Since $n_m \le R+1$, the largest part $V_m$ can be interleaved with the vertices from the remaining parts. More precisely, list the vertices of $V_m$ as
\[
v_1,\dots,v_{n_m},
\]
and list all vertices of $\bigcup_{i<m} V_i$ arbitrarily as
\[
w_1,\dots,w_R.
\]

Since $n_m \le R+1$, we first arrange them in the alternating sequence
\[
v_1, w_1, v_2, w_2, \dots, v_{n_m}, w_{n_m}.
\]

Now, to complete the path, repeatedly pick a vertex from the largest remaining part that wasn't used last. 
The above argument is equivalent to saying that there exists an ordering of all vertices such that no two consecutive vertices lie in the same part.

This follows from the standard necessary and sufficient condition for rearranging a multiset so that no two identical elements are adjacent.

Hence $\lambda(G) + 1 = S$.
\medskip

\noindent\textbf{Case 2: $n_m > R+1$.}

We show that the bound $2R+1$ is sharp.

Choose all vertices from $\bigcup_{i<m} V_i$ and exactly $R+1$ vertices from $V_m$. Arrange them alternately as
\[
v_1, w_1, v_2, w_2, \dots, v_R, w_R, v_{R+1},
\]
where $v_i \in V_m$ and $w_i \in \bigcup_{i<m} V_i$.

This gives a path with $2R+1$ vertices.

Since we have already shown that no path can have more than $2R+1$ vertices, it follows that
\[
\lambda(G) + 1 = 2R+1.
\]
Combining the two cases yields the result.
\end{proof}

\begin{lemma}\label{lem: faces for CMG}
Let $G=K_{n_1, n_2, \dots, n_m}$ be a complete multipartite graph with vertex partition
\[
V = V_1 \sqcup V_2 \sqcup \cdots \sqcup V_m,
\]
and let $F \subseteq V$ be such that $|F \cap V_i| = a_i$ for each $i \in [m]$.

Then $F \in \PF_t(G)$ if and only if one of the following holds:
\begin{enumerate}
    \item $|F| < t$.

    \item $|F| \ge t$, and there do not exist integers $a_1', \dots, a_m'$ such that:
    \begin{itemize}
        \item $0 \le a_i' \le a_i$ for all $i \in [m]$,
        \item $\displaystyle \sum_{i=1}^m a_i' = t,$
        \item if $j \in [m]$ satisfies $a_j' = \max\{a_i' : i \in [m]\}$, then
        \[
        a_j' - \sum_{i \neq j} a_i' \le 1.
        \]
    \end{itemize}
\end{enumerate}
\end{lemma}
\begin{proof}
If $|F|<t$ then $G[F]$ cannot contain a path on $t$ vertices.
Hence $F\in\PF_t(G)$.

We show that, for $|F|\geq t$,  
\[
F\notin\PF_t(G)
\iff
\text{there exist }a_1',\dots,a_m'\text{ as in (2)}.
\]

($\Rightarrow$) Suppose $F\notin\PF_t(G)$. Then $G[F]$ contains a copy of
$P_t$. Let $U\subseteq F$ be the set of vertices of such a path, and define
$a_i':= |U\cap V_i|$.
Then $0\le a_i'\le a_i$ for all $i$, and $\sum_{i=1}^m a_i'=|U|=t$.

Let $j\in[m]$ with $a_j' = \max\{a_i' : i\in[m]\}$.
Along a path, two consecutive vertices cannot lie in the same part of a complete
multipartite graph. Therefore, each vertex of $U\cap V_j$, except possibly one,
must be separated by a vertex of $U\setminus V_j$. Hence
\[
a_j'\le |U\setminus V_j|+1
=\sum_{i\neq j}a_i'+1,
\]
equivalently,
\begin{equation}
a_j'-\sum_{i\neq j}a_i'\le 1. \label{inequality1}
\end{equation}
So the required integers $a_1',\dots,a_m'$ exist.

($\Leftarrow$) Conversely, suppose there exist integers $a_1',\dots,a_m'$
satisfying the conditions in (2). For each $i\in[m]$, choose a subset
$W_i\subseteq F\cap V_i$ with $|W_i|=a_i'$, and set
\[
U:=\bigsqcup_{i=1}^m W_i.
\]
Then $|U|=\sum_i a_i'=t$ and
\[
G[U]\cong K_{a_1',\dots,a_m'}.
\]

Relabel indices so that $a_m' = \max\{a_i' : i\in[m]\}$. The inequality in (2) gives
\[
a_m'\le \sum_{i=1}^{m-1}a_i'+1.
\]
By \Cref{lem: number of vertices in a longest path} and the inequality \eqref{inequality1}, the graph $G[U]$ contains a path on $t$ vertices. 

Hence $G[F]$ contains a $P_t$ and therefore $F\notin\PF_t(G)$.

Thus, for $|F|\ge t$, the set $F$ lies in $\PF_t(G)$ if and only if no tuple $(a_1',\dots,a_m')$ satisfying the above conditions exists.
\end{proof}

We now translate the path-avoidance condition into a simplicial decomposition condition. 
Before that, we prove a technical result, which, as a numerical reformulation, is the key consequence of the previous Lemma.

\begin{lemma}\label{lemma: key tech}
Let $F \subset V$ with $a_i = |F \cap V_i|,\, s = |F|, \text{~and} M = \max_i a_i$.
Then $G[F]$ contains a path on $t$ vertices if and only if
$t \leq 2(s - M) + 1$.
\end{lemma}

\begin{proof}
Let $F \subseteq V$, and for each $i \in [m]$ set
\[
a_i = |F \cap V_i|,\qquad
s = |F|,\qquad
M = \max_i a_i.
\]
We prove both directions by extracting quantitative constraints from the path-existence criterion established earlier. 
Throughout, let $j\in [m]$ denote any index achieving $a_j'=\max_i a_i'$ in the relevant integer tuple.

\smallskip

Suppose first that $G[F]$ contains a path on $t$ vertices.
By \Cref{lem: faces for CMG}, there exist integers $a_1',\dots,a_m'$ such that
\[
0 \le a_i' \le a_i,\qquad
\sum_{i=1}^m a_i' = t,
\]
and if $j$ satisfies
\[
a_j' = \max_i a_i',
\]
then
\[
a_j' - \sum_{i\neq j} a_i' \le 1.
\]
Since $\sum_{i\neq j} a_i' = t-a_j'$, the path-alternation condition from the previous Lemma gives $2a_j'\leq t+1$. 
On the other hand, since $a_j'\leq a_j\leq M$ and $\sum_{i\neq j}a_i'\leq \sum_{i\neq j} a_i = s - M$, we get $a_j' = t-\sum_{i\neq j} a_i'\geq t-(s-M)$. 
Combining the lower and upper bounds on $a_j'$:
\[
t-(s-M) \leq a_j' \leq (s-M) +1,
\]
and subtracting gives $t\leq 2(s-M)+1$. 

\smallskip

Now for the converse, choose an index $j$ such that $a_j=M$, and set $B := s-M = \sum_{i\neq j} a_i$.
Then $t \le 2B+1$. 
We now construct integers $a_1',\dots,a_m'$ satisfying the conditions of \Cref{lem: faces for CMG}.

Choose integers $a_i'$ for $i\neq j$ such that $0 \le a_i' \le a_i$ and 
$\sum_{i\neq j} a_i' = \min(B,t-1)$.

Define
\[
a_j' := t-\sum_{i\neq j} a_i'.
\]
Then $\sum_{i=1}^m a_i' = t$.
Since $\sum_{i\neq j} a_i' \le B$,
we have $a_j' \ge t-B$. 
Also,
\[
a_j'
=
t-\sum_{i\neq j} a_i'
\le
t-(t-1)=1
\]
when $B\ge t-1$, while if $B<t-1$, then
\[
a_j'=t-B \le B+1
\]
because $t\le 2B+1$.

Thus in all cases,
\[
a_j' \le \sum_{i\neq j} a_i' +1.
\]

Finally, since
\[
a_j'=t-\sum_{i\neq j} a_i' \le t-(t-1)=1 \le M=a_j
\]
or
\[
a_j'=t-B \le M
\]
according as $B\ge t-1$ or $B<t-1$, we have
$0\le a_j' \le a_j$.

Hence the integers $a_1',\dots,a_m'$ satisfy all conditions of \Cref{lem: faces for CMG}. 
Therefore $G[F]$ contains a path on $t$ vertices.
\end{proof}

\subsection{Structural Decomposition}

We now use this result to derive a simplicial decomposition of the path-free complex. 
The union of join complexes captures all faces satisfying the above inequality, while the remaining faces arise purely from the size constraint.

\begin{lemma}\label{lem:decomposition-CMG}
Let $G$ be the complete multipartite graph introduced above. For each $j$, set
\[
W_j = \bigcup_{i \ne j} V_i.
\]
Then the $t$-path-free complex $\PF_t(G)$ satisfies
\begin{equation}\label{eq: structural decomposition}
\PF_t(G)=
\left(
\bigcup_{j=1}^{m}
\left(
\Delta_{V_j} * (\Delta_{W_j})^{[k-2]}
\right)
\right)
\cup
(\Delta_V)^{[t-2]}.
\end{equation}
\end{lemma}

\begin{proof}
Let $F \subseteq V(G)$. Set $s = |F|, \; M := \max_i |F\cap V_i|$ as before. 
Then $F \in \PF_t(G)$ if and only if
\[
s < t \quad \text{or} \quad 2(s-M)+1 < t.
\]
The second inequality
is equivalent to
\[
s - M \le k-1.
\]

Hence
\[ F \in \PF_t(G)
\quad \Longleftrightarrow \quad
s < t \; \text{or} \; s - M \le k-1.
\]

\medskip

\noindent\textbf{(1) The condition $s - M \le k-1$.}

Choose $j$ such that $a_j = M$. Then
\[
s - M = \sum_{i \ne j} a_i = |F \cap W_j|.
\]
Therefore
\[
s - M \le k-1
\quad \Longleftrightarrow \quad
|F \cap W_j| \le k-1.
\]

This means that $F \cap W_j$ is a face of the $(k-2)$-skeleton $(\Delta_{W_j})^{[k-2]}$,
while $F \cap V_j$ is arbitrary. Hence
\[
\F \in \Delta_{V_j} * (\Delta_{W_j})^{[k-2]}.
\]

Since such an index $j$ exists whenever $s-M \le k-1$, we obtain
\[
\{F : s - M \le k-1\} = \bigcup_{j=1}^m
\bigl(
\Delta_{V_j} * (\Delta_{W_j})^{[k-2]}
\bigr).
\]

\medskip

\noindent\textbf{(2) The condition $s < t$.}

This condition depends only on the size of $F$, and is equivalent to
\[
F \in (\Delta_V)^{[t-2]}.
\]

Since $F\in \PF_t(G)$ if and only if at least one of the conditions $s<t$ or $s-M\leq k-1$ holds, the path-free complex is the union of the two families identified in (1) and (2) above, giving the stated decomposition.
\end{proof}

\begin{remark}\label{rem: dim}
Since
$
|V_j| \le \left\lceil \frac{t}{2} \right\rceil,
$
we have
$
|V_j|+k-2
\le
\left\lceil \frac{t}{2} \right\rceil+k-2
\le t-2.
$
Hence, whenever
$
|V_j| \le \left\lceil \frac{t}{2} \right\rceil,
$
the complex
$
\Delta_{V_j} * (\Delta_{W_j})^{[k-2]}
$
has dimension at most $t-2$, and therefore,
\[
\Delta_{V_j} * (\Delta_{W_j})^{[k-2]}
\subseteq
(\Delta_V)^{[t-2]}.
\]

Thus, only those parts $V_j$ with
$
|V_j| > \left\lceil \frac{t}{2} \right\rceil
$
contribute simplices not already contained in
$
(\Delta_V)^{[t-2]}.
$
Let
$$
A=\left\{j\in[m] : |V_j| > \left\lceil \frac{t}{2} \right\rceil\right\}.
$$
Then
$$
\PF_t(G)=
\left(
\bigcup_{j\in A}
\bigl(
\Delta_{V_j} * (\Delta_{W_j})^{[k-2]}
\bigr)
\right)
\cup
(\Delta_V)^{[t-2]}.
$$

Consequently,
\[
\dim(\PF_t(G))=
\begin{cases}
t-2, & \text{if } A=\emptyset,\\[4pt]
|V_j|+k-2, & \text{if } A\neq\emptyset,\ \text{where } j=\max A.
\end{cases}
\]

In particular, the above description yields the following consequences.
\begin{enumerate}
    \item If $t\ge 2n_m-1$, then \(A=\emptyset\), and hence
    \[
    \PF_t(G)=(\Delta_V)^{[t-2]}.
    \]

    \item If $t\ge 2n_{m-1}-1
    \quad\text{and}\quad
    n_{m-1}\neq n_m$, then \(A=\{m\}\). Consequently,
    \[
    \PF_t(G)=
    \bigl(
    \Delta_{V_m} * (\Delta_{W_m})^{[k-2]}
    \bigr)
    \cup
    (\Delta_V)^{[t-2]}.
    \]
\end{enumerate}

\end{remark}

\subsection{Sequential Cohen–Macaulayness and Vertex Decomposability}
We now analyze Cohen-Macaulay properties of $\PF_t(G)$.

\begin{lemma}\label{lem: SCM for CMG}
    If $t\le 2n_{m-1}-2$, then $\PF_t(G)$ is not sequentially Cohen-Macaulay.
\end{lemma}
\begin{proof}
Let $\Delta = \bigl(\PF_t(G)\bigr)^{[t-1]}$ be the pure $(t-1)$-skeleton. 

\medskip
Let 
\[
F=\{v^{(m-1)}_1,\ldots,v^{(m-1)}_{k-1}, v^{(m)}_1,\ldots,v^{(m)}_{k-1}\}.
\]
Then
\[  |F\cap V_{m-1}| = |F\cap V_m| = k-1. \]

The induced subgraph $G[F]$ is the complete bipartite graph $K_{k-1,k-1}$, whose longest path has $2(k-1)$ vertices. Since $2k-2<t$, we have $F\in\PF_t(G)$. We can enlarge $F$ to a set $F'\supseteq F$ with $|F'|=t$ by adding vertices from either $V_m\setminus F$ or $V_{m-1}\setminus F$ so that the induced subgraph $G[F']$ is a complete bipartite graph whose longest path has $2k-1$ vertices. Hence, $F'\in\PF_t(G)$ and is a facet of $\Delta$. Therefore, $F\in\Delta$.

Thus $\lk_{\Delta}(F)$ is a pure complex of dimension $t-2k+1$. 

Let
\[
A = V_{m-1} \setminus F, \quad B = V_m \setminus F.
\]

A set $F' \subseteq V \setminus F$ is a face of $\mathrm{lk}_\Delta(F)$ if and only if
$F \cup F' \in \Delta$, which is equivalent to $|F \cup F'| \le t$ and it contains in a face of dimension $t-1$. 
Since
$|F| = 2(k-1)$, this gives $|F'| \le t - 2k + 2$.

Moreover, if $F'$ contains vertices from both $A$ and $B$ or a vertex from other part, then $G[F \cup F']$
contains a path on $t$ vertices, so $F'$ does not lie in the link. Hence
\[
\mathrm{lk}_\Delta(F)
=
(\Delta_A)^{[t-2k+1]} \;\cup\; (\Delta_B)^{[t-2k+1]}.
\]

Since $t \le 2n_{m-1} - 2$, we have $k-1 \le n_{m-1} - 2$, so both $A$ and $B$
are nonempty. Thus $\mathrm{lk}_\Delta(F)$ is a disconnected simplicial complex
of dimension at least $1$, and hence
\[
\widetilde{H}_0(\mathrm{lk}_\Delta(F)) \neq 0.
\]

Therefore the pure skeleton $\Delta$ is not Cohen–Macaulay, and as a result the complex $\PF_t(G)$ is not
sequentially Cohen–Macaulay.
\end{proof}

\begin{theorem}\label{them: General V.D}
Let $V$ be a vertex set and let $W \subseteq V$. Suppose that
\[
-1 \le a \le |W^{c}|-1
\qquad \text{and} \qquad
-1 \le b \le |V|-1.
\]
Then the simplicial complex
\[
\Sigma
=
\left(\Delta_{W} * (\Delta_{W^{c}})^{[a]} \right)
\cup
(\Delta_{V})^{[b]}
\]
is vertex decomposable.
\end{theorem}

\begin{proof}
Suppose $b \ge a+|W|$. Every face of $\Delta_W * (\Delta_{W^c})^{[a]}$ has dimension at most $a+|W|$, so it is a face of $(\Delta_V)^{[b]}$. Thus $\Delta_W * (\Delta_{W^c})^{[a]}$ is a subcomplex of $(\Delta_V)^{[b]}$, and hence $\Sigma=(\Delta_V)^{[b]}$. Since the pure skeleton of simplices is vertex decomposable, $\Sigma$ is vertex decomposable.

Thus we may assume that
\[
b<a+|W|.
\]

If $a=|W^c|-1$, then $(\Delta_{W^c})^{[a]}=\Delta_{W^c}$, and hence $\Delta_W * \Delta_{W^c} = \Delta_V$. Similarly, if $b=|V|-1$, then $(\Delta_V)^{[b]}=\Delta_V$. Also, if $W = V$ then $\Sigma = \Delta_V$. In all cases, $\Sigma=\Delta_V$, and therefore $\Sigma$ is vertex decomposable.

Hence we may further assume that $W \subset V$, 
\[
-1\le a<|W^c|-1
\qquad \text{and} \qquad
-1\le b<|V|-1.
\]

We now prove the statement by induction on $|V|$.

\medskip

\noindent
\textbf{Base case:}
If $|V|=1$, then $\Sigma$ is a simplex. Hence $\Sigma$ is vertex decomposable.

\medskip
\noindent
\textbf{Induction hypothesis:}
Assume that the statement holds for all graphs with the number of vertices strictly smaller than $n$.

\medskip

\noindent
\textbf{Induction step:}
Suppose $|V|=n$.

Choose a vertex $v \in W^c$. 
Such a vertex exists since $W \subset V$ implies that $W^c\neq \emptyset$.

We show that $v$ is a shedding vertex.

First consider the deletion:
\[
\del_\Sigma(v)
=
\left(
\Delta_W * (\Delta_{W^c\setminus\{v\}})^{[a]}
\right)
\cup
(\Delta_{V\setminus\{v\}})^{[b]}.
\]

Since $a<|W^c|-1$, we obtain $a\le |W^c\setminus\{v\}|-1$. 
Also, $b\le |V\setminus\{v\}|-1$, because $|V\setminus\{v\}|=|V|-1$.

Therefore $\del_\Sigma(v)$ is of the same form as in the statement, but on the smaller vertex set $V\setminus\{v\}$. 
By the induction hypothesis, the deletion $\del_\Sigma(v)$ is vertex decomposable.

Next consider the link:
\[
\lk_\Sigma(v)
=
\left(
\Delta_W * (\Delta_{W^c\setminus\{v\}})^{[a-1]}
\right)
\cup
(\Delta_{V\setminus\{v\}})^{[b-1]}.
\]

If $a=-1$, then
($(\Delta_{W^c})^{[a]}=\emptyset$, and hence $\lk_\Sigma(v)=(\Delta_{V\setminus\{v\}})^{[b-1]}$.) Otherwise, $-1\le a-1\le |W^c\setminus\{v\}|-1$.

In either case, $\lk_\Sigma(v)$ is of the same form as in the statement on the vertex set $V\setminus\{v\}$. 
Since $-1\le b-1\le |V\setminus\{v\}|-1$, the induction hypothesis implies that $\lk_\Sigma(v)$ is vertex decomposable.

It remains to verify the shedding condition.
Let $F$ be a facet of $\lk_\Sigma(v)$. 
Then either $F=W\cup F'$, where $F'$ is a facet of  $(\Delta_{W^c\setminus\{v\}})^{[a-1]}$, or $F$ is a facet of $(\Delta_{V\setminus\{v\}})^{[b-1]}$

First suppose $F=W\cup F'$. Since $F'$ is a facet of
\[
(\Delta_{W^c\setminus\{v\}})^{[a-1]},
\]
we have $|F'|=a$, so $|F|=|W|+a$. Because $b<a+|W|$, we get $|F|>b$, hence $F$ is not a face of $(\Delta_{V\setminus\{v\}})^{[b]}$.

The inequality $a<|W^c|-1$ yields a vertex $w\in W^c\setminus(F'\cup\{v\})$. Then $F'\cup\{w\}$ is a face of $(\Delta_{W^c\setminus\{v\}})^{[a]}$, so $F\cup\{w\}=W\cup(F'\cup\{w\})$ is a face of $\del_\Sigma(v)$. Thus $F$ is not a facet of $\del_\Sigma(v)$.

Now suppose $F$ is a facet of $(\Delta_{V\setminus\{v\}})^{[b-1]}$. Then $|F|=b$. Since $b<|V|-1$, there exists $w\in V\setminus(F\cup\{v\})$, and $F\cup\{w\}$ is a face of $(\Delta_{V\setminus\{v\}})^{[b]}$, hence of $\del_\Sigma(v)$. Thus $F$ is not a facet of $\del_\Sigma(v)$.

Therefore no facet of $\lk_\Sigma(v)$ is a facet of $\del_\Sigma(v)$, so $v$ is a shedding vertex. Since both $\del_\Sigma(v)$ and $\lk_\Sigma(v)$ are vertex decomposable, $\Sigma$ is vertex decomposable.
\end{proof}

\begin{corollary}\label{Cor: V.D for particular case}
For $t \ge 2n_{m-1}-1$, the $t$-path free complex $\PF_t(G)$ is vertex decomposable. 
\end{corollary}
\begin{proof}
If $t \ge 2n_{m-1}-1$, then, by Remark~\ref{rem: dim},
\[
\PF_t(G)
=
\bigl(\Delta_{V_m} * (\Delta_{W_m})^{[k-2]}\bigr)
\cup
(\Delta_V)^{[t-2]}.
\]
Therefore, Theorem~\ref{them: General V.D} implies that $\PF_t(G)$ is vertex decomposable.

\end{proof}

\begin{theorem}\label{them: char for CMG}
    Let $G=K_{n_1, n_2, \dots, n_m}$ and let $t \ge 2$. Then the following statements are equivalent:
    \begin{enumerate}
        \item $t \ge 2n_{m-1}-1$.
        \item $\PF_t(G)$ is vertex decomposable.
        \item $\PF_t(G)$ is shellable.
        \item $\PF_t(G)$ is sequentially Cohen-Macaulay.
    \end{enumerate}
\end{theorem}
\begin{proof}
The result follows from Equation~\ref{eq: implications}, Lemma~\ref{lem: SCM for CMG} and Corollary~\ref{Cor: V.D for particular case}
\end{proof}

\section{Determining the Exact Homotopy Type}\label{sec:CBG-structure-topology}

One of the main topological implications of shellable complexes is that they have the homotopy type of a wedge of spheres \cite[Theorem 12.3]{kozlov2007}. 
Hence the natural question now is to determine whether the path-free complexes of complete multipartite graphs are homotopic to wedges of spheres?
We address this question in the present Section. 
The decomposition established in \Cref{lem:decomposition-CMG} expresses the complex as a union of contractible join complexes with controlled intersections, but the topology is considerably richer. 

In contrast to the general multipartite setting, the decomposition becomes particularly tractable for complete bipartite graphs and when $t=3$. 
In both these cases, the homotopy type is that of a wedge of spheres. 

We begin with complete bipartite graphs, as a first step we convert the path-avoidance conditions into explicit join decompositions of $\PF_t(G)$.

Let $G=K_{m,n}$ be the complete bipartite graph with bipartition
\[
X=\{x_1,\dots,x_m\}, \qquad Y=\{y_1,\dots,y_n\}, \qquad m\le n.
\]

Since every path in $K_{m,n}$ alternates between the two parts, the longest path uses all vertices of the smaller part $X$ and as many vertices of $Y$ as possible. 
Using \Cref{lem: number of vertices in a longest path} we have the formula for the length of the longest path. 

\begin{equation}\label{max-path-length in CBG}
\lambda(K_{m,n})=
\begin{cases}
2m, & \text{if } m<n,\\[2pt]
2m-1, & \text{if } m=n.
\end{cases}
\end{equation}
In particular, by \Cref{path free for max path}, $\PF_t(G)$ is a simplex whenever $t>\lambda(K_{m,n})+1$.

We begin by stating the join decomposition of the path-free complex.

\begin{lemma}\label{lem: structure of the complex CBG1}
Let $k=\lfloor t/2\rfloor$.
\begin{enumerate}
    \item If $t\le 2m$ is even, then
    \[
    \PF_t(G)
    =
    (\Delta_X * (\Delta_{Y})^{[k-2]})
    \cup
    ((\Delta_X)^{[k-2]} * \Delta_Y).
    \]
    \item If $t\le 2m$ is odd, then
    \[
    \PF_t(G)
    =
    (\Delta_X * (\Delta_{Y})^{[k-2]})
    \cup
    ((\Delta_{X})^{[k-1]} * (\Delta_{Y})^{[k-1]})
    \cup
    ((\Delta_{X})^{[k-2]} * \Delta_Y).
    \]
\end{enumerate}
\end{lemma}

\begin{proof}
Let $F\subseteq X\sqcup Y$ such that $a=|F\cap X|$ and $b=|F\cap Y|$.
Since the induced subgraph $G[F]$ is the complete bipartite graph $K_{a,b}$, every path in $G[F]$
must alternate between the two bipartition classes.

Suppose first that $t=2k$.
A path on $2k$ vertices alternates between the two parts and therefore uses
exactly $k$ vertices from each part.
Hence $K_{a,b}$ contains a copy of $P_{2k}$ if and only if
$a\ge k$ and $b\ge k$.
Therefore
\[
F\in PF_t(G)
\iff
(a<k)\ \text{or}\ (b<k),
\]
which is equivalent to
\[
a\le k-1 \quad\text{or}\quad b\le k-1.
\]

Now suppose that $t=2k+1$.
A path on $2k+1$ vertices alternates between the two parts and therefore uses
$k+1$ vertices from one part and $k$ vertices from the other.
Consequently $K_{a,b}$ contains a copy of $P_{2k+1}$ if and only if
\[
(a\ge k+1,\ b\ge k)
\quad\text{or}\quad
(a\ge k,\ b\ge k+1).
\]

Negating this condition, we obtain that $F\in PF_t(G)$ if and only if
\[
a\le k-1,
\]
or
\[
b\le k-1,
\]
or else both $a,b\ge k$ but neither exceeds $k$.
The latter possibility is precisely
\[
a=b=k.
\]

Hence
\[
F\in PF_t(G)
\iff
a\le k-1
\ \text{or}\
b\le k-1
\ \text{or}\
(a,b)=(k,k).
\qedhere \]
\end{proof}

Note that the odd case differs from the even case only through the balanced configuration $a=b=k$, which contributes the additional term. We treat the case $t=2m+1$ separately as follows. 

\begin{lemma}\label{lem: facets for t=2m+1}
For  $t=2m+1$ we have 
\[
\PF_{2m+1}(G)
=
(\Delta_X * (\Delta_{Y})^{[m-1]})
\cup
((\Delta_{X})^{[m-2]} * \Delta_Y).
\]
\end{lemma}

\begin{proof}
Applying \Cref{lem:decomposition-CMG} with $k=m$, we obtain
\[
\PF_{2m+1}(G)
=
(\Delta_X * (\Delta_Y)^{[m-2]})
\cup
((\Delta_X)^{[m-1]} * (\Delta_Y)^{[m-1]})
\cup
((\Delta_X)^{[m-2]} * \Delta_Y).
\]

Since $|X|=m$, the $(m-1)$-skeleton of $\Delta_X$ is precisely
the $\Delta_X$. Thus
\[
(\Delta_X)^{[m-1]}
=
\Delta_X.
\]

Substituting this into the above decomposition gives
\[
\PF_{2m+1}(G)
=
(\Delta_X * (\Delta_Y)^{[m-2]})
\cup
(\Delta_X * (\Delta_Y)^{[m-1]})
\cup
((\Delta_X)^{[m-2]} * \Delta_Y).
\]

Now

\[
(\Delta_Y)^{[m-2]}
\subseteq
(\Delta_Y)^{[m-1]}.
\]

Therefore
\[
\Delta_X * (\Delta_Y)^{[m-2]}
\subseteq
\Delta_X * (\Delta_Y)^{[m-1]}.
\]

Hence
\[
(\Delta_X * (\Delta_Y)^{[m-2]})
\cup
(\Delta_X * (\Delta_Y)^{[m-1]})
=
(\Delta_X * (\Delta_Y)^{[m-1]}),
\]
and consequently
\[
\PF_{2m+1}(G)
=
(\Delta_X * (\Delta_{Y})^{[m-1]})
\cup
((\Delta_{X})^{[m-2]} * \Delta_Y).
\]
This completes the proof.
\end{proof}

Now we prove the main result of this Section. 

\begin{theorem}\label{homotopy-decomposition-cbg}
Let $G = K_{m,n}$ be a complete multipartite graph. For all $2\leq t\leq \lambda(K_{m,n})+1$ the path-free complexes 
$\PF_t(G)$ has the homotopy type of a wedge of spheres. 
\end{theorem}

\begin{proof} First, we consider the even case. 
Let $2\leq t=2k\leq 2m$ be even. 
From Lemma~\ref{lem: structure of the complex CBG1} we know that the path-free complex is the union of the following two subcomplexes 
\[
    A =  (\Delta_X * (\Delta_{Y})^{[k-2]}),
    \quad
    B = ((\Delta_X)^{[k-2]} * \Delta_Y).
\]
Note that both $ A$ and $ B$ are contractible. 
Hence, by the homotopy pushout property for the union of two contractible subcomplexes \cite[Section 15.2]{kozlov2007},
\[A \cup B \simeq \Sigma (A \cap B). \]
The intersection of these subcomplexes is equal to
$(\Delta_X)^{[k-2]}\ast (\Delta_Y)^{[k-2]}. $
Recall that $i$-skeleton of an $l$-simplex is homotopic to a wedge $\binom{l}{i+1}$ many $i$-dimensional spheres. 
Using this in our context, we get
\begin{align*}
    \PF_{2k}(G) &\simeq \Sigma((\Delta_X)^{[k-2]}\ast (\Delta_Y)^{[k-2]})\\
        &\simeq \Sigma(\bigvee_{\binom{m-1}{k-1}} S^{k-2}\ast \bigvee_{\binom{n-1}{k-1}}S^{k-2}) \\
        &\simeq \Sigma(\bigvee_{\binom{m-1}{k-1}\binom{n-1}{k-1}} S^{2k-3})\\
        &\simeq \bigvee_{\binom{m-1}{k-1}\binom{n-1}{k-1}} S^{2k-2}.
\end{align*}
The case $t=2m+1$ is similar, since the path-free complex is a union of two contractible subcomplexes.  
We have,
    \[
    \PF_{2m+1}(G)
    \simeq
    \bigvee_{\binom{n-1}{m}} S^{2m-1}.
    \]
Finally, the odd case $2  < t = 2k+1 < 2m $. 
First note that every subset of size at most $t-1$ is automatically a face.
This means that the complex has the full $(t-2)$-dimensional skeleton.
Equivalently $\PF_{2k+1}(G)$ is $(t-3) = (2k-2)$-connected. 
Since $2k+1$ is at least $3$, in the odd case the path-free complex is always connected. 
We start with the smallest possible value, $t=3$. 
Using the join decomposition of Lemma~\ref{lem: structure of the complex CBG1} we get
\begin{align*}
    \PF_3(K_{m,n}) &= (\Delta_X\ast\emptyset)\cup((\Delta_X)^{[0]}\ast(\Delta_Y)^{[0]})\cup (\emptyset\ast\Delta_Y) \\
    &= \Delta_X \cup (K_{m,n}) \cup \Delta_Y\\
    &\simeq \bigvee_{mn-1} S^1. 
\end{align*}
Note that the simplices on $X, Y$ are contractible, hence can be homotoped to a point. 
We are left with two vertices and $mn$ edges between them. 

Finally, we consider the remaining values of $t$, where the path-free complexes are simply connected. 
For notational simplicity we write $A = (\Delta_X * (\Delta_{Y})^{[k-2]})$, $B = ((\Delta_{X})^{[k-2]} * \Delta_Y)$ and $M = ((\Delta_{X})^{[k-1]} * (\Delta_{Y})^{[k-1]})$. In this situation $M$ is not contractible, so we will use the Mayer-Vietoris spectral sequence (see \cite[Example 14.16]{BottTu1982} for more details). 
In order to set up the first page of this spectral sequence, we need to determine the homotopy type of each nonempty intersection of these spaces. 
Clearly $A, B$ are contractible and $M\simeq \bigvee_{\binom{m-1}{k}\binom{n-1}{k}}S^{2k-1}$. The homotopy type of intersections is given below
\begin{align*}
    A\cap B &\simeq \bigvee_{\binom{m-1}{k-1}\binom{n-1}{k-1}}S^{2k-3}\\
    A\cap M &\simeq \bigvee_{\binom{m-1}{k}\binom{n-1}{k-1}}S^{2k-2}\\
    B\cap M &\simeq \bigvee_{\binom{m-1}{k-1}\binom{n-1}{k}}S^{2k-2}.
\end{align*}
Finally, note that $A\cap B\cap M = A\cap B$. 
Hence, the only nonzero terms on the first page are as follows
\begin{align*}
    E^1_{0, 2k-1} &\cong \Tilde{H}_{2k-1}(M)\\
    E^1_{1, 2k-3} &\cong \Tilde{H}_{2k-3}(A\cap B)\\
    E^1_{1, 2k-2} &\cong \Tilde{H}_{2k-2}(A\cap M)\oplus \Tilde{H}_{2k-2}(B\cap M) \\
    E^1_{2, 2k-3} &\cong \Tilde{H}_{2k-3}(A\cap B\cap M)
\end{align*}
The only nonzero differential is $d_{2, 2k-3}: E^1_{2, 2k-3}\to E^1_{1, 2k-3}$, which is an isomorphism. 
Hence, the second page has the following nonzero entries:
\[E^2_{0, 2k-1} = E^1_{0, 2k-1},\quad E^2_{1, 2k-2} = E^1_{1, 2k-2}, \] and the other differentials are zero. 
Therefore, 2nd page is the infinity page and the spectral sequence abuts to the homology of the union. 
The total degree is $2k-1$ and the associated graded of $\Tilde{H}_{2k-1}(\PF_{2k+1})$ has exactly two pieces: $E^\infty_{0, 2k-1}$ and $E^\infty_{1, 2k-2}$. 
Since both of them are free abelian, the extension problem is trivial, and we have
\begin{align*}
    \Tilde{H}_{2k-1}(K_{m,n}) &= \Tilde{H}_{2k-1}(M)\oplus\Tilde{H}_{2k-2}(A\cap M)\oplus \Tilde{H}_{2k-2}(B\cap M)\\
    &= \bigoplus_{\binom{m-1}{k}\binom{n-1}{k}} \Z \oplus \left(\bigoplus_{\binom{m-1}{k}\binom{n-1}{k-1}} \Z\right) \oplus \left(\bigoplus_{\binom{m-1}{k-1}\binom{n-1}{k}} \Z\right).
\end{align*}

We have that the complex $\PF_{2k+1}(K_{m,n})$ is simply connected, it reduced homology vanishes except in degree $2k-1$. 
Moreover, in that degree the homology group is free abelian. 
Hence, by an application of Hurewicz' theorem and Whitehead's theorem (see \cite[Section 6.7]{kozlov2007}), we get that 
\[\PF_{2k+1}(K_{m,n}) \simeq \bigvee_{\beta} S^{2k-1},\]
where $\beta = \binom{m-1}{k}\binom{n-1}{k} + \binom{m-1}{k}\binom{n-1}{k-1} + \binom{m-1}{k-1}\binom{n-1}{k}$. 
This covers all the cases, and the theorem follows. 
\end{proof}

We now determine the homotopy type of $\PF_3(K_{n_1,\ldots,n_r})$.
The decomposition obtained from \Cref{lem:decomposition-CMG} enables an explicit computation of its homotopy type. 
In addition, we apply the \emph{multinerve theorem} (see \cite{multinerve01, ramras}) to deduce that this homotopy type is that of a wedge of circles.
Recall that the classical nerve theorem asserts that if a topological space admits a good open cover, then the nerve of this cover has the same homotopy type as the space itself. 
A cover is called good if each open set is contractible and every nonempty finite intersection is also contractible. 
In this setting, the nerve is a simplicial complex. 
The theorem has been generalized to encompass \emph{quasi-good} covers, in which intersections are allowed to be disconnected, under the condition that each of their connected components is contractible. 
In this more general situation, the nerve is in general a $\Delta$-complex. 
We now state precisely the version of the theorem that we will use.

\begin{theorem}[Corollary 3.10 of \cite{multinerve02}]
    Let $K$ be a simplicial complex and let $\mathcal{U} = \{L_i: i\in I\}$ be a quasi-good cover of $K$ by subcomplexes. Then the complex $K$ is (simply) homotopy equivalent to the nerve of $\mathcal{U}$. 
\end{theorem}

Now we explicitly characterize the homotopy type of $3$-path-free complexes.

\begin{theorem}
Let
\[
G=K_{n_1,n_2,\ldots,n_r},
\]
where $n_i\ge 1$ for all $1\le i\le r$. Then
\[
PF_3(G)
\simeq
\bigvee^{\,\sum_{1\le i<j\le r}n_in_j-r+1}S^1.
\]
\end{theorem}

\begin{proof}
By Lemma~\ref{lem:decomposition-CMG},
\[
PF_3(G)
=
(\Delta_V)^{[1]}
\;\cup\;
\bigcup_{i=1}^r\Delta_{V_i},
\]
where every simplex of dimension at least two is contained entirely in one of the simplices
$\Delta_{V_1},\ldots,\Delta_{V_r}$.

Perform a stellar subdivision (\cite[Definition 2.22]{kozlov2007}) of every mixed edge (that is, every edge joining two distinct parts) by introducing one new vertex. 
Since no mixed edge belongs to a simplex of dimension greater than one, each stellar subdivision is simply an edge subdivision, and these subdivisions are independent. 
Let $K'$ denote the resulting simplicial complex. Clearly,
$K'$ is a subdivision of $PF_3(G)$ and hence
\[
|K'|\cong |PF_3(G)|.
\]

For each $1\le i\le r$, let $L_i$ be the subcomplex of $K'$ consisting of
\begin{itemize}
    \item the simplex $\Delta_{V_i}$, and
    \item every subdivided half-edge having one endpoint in $V_i$.
\end{itemize}

We first verify the hypotheses of the multinerve theorem.

\smallskip

\noindent
\textbf{(i) Each $L_i$ is contractible.}
Every subdivided half-edge contains a subdivision vertex which is a free vertex in the subcomplex $L_i$. 
Consequently, these half-edges may be removed one at a time by elementary simplicial collapses, leaving only the simplex $\Delta_{V_i}$. Hence $L_i$ simplicially collapses onto $\Delta_{V_i}$ and is therefore contractible.

\smallskip

\noindent
\textbf{(ii) Pairwise intersections.}
For $i\neq j$, the intersection
$L_i\cap L_j$
consists precisely of the subdivision vertices of the mixed edges joining $V_i$ and $V_j$. Since there are exactly $n_in_j$ such edges,
$L_i\cap L_j$
has exactly $n_in_j$ connected components, each consisting of a single vertex. In particular, every connected component is contractible.

\smallskip

\noindent
\textbf{(iii) Higher intersections.}
Every subdivision vertex lies on exactly one subdivided mixed edge. Therefore no subdivision vertex belongs to three distinct subcomplexes, and hence
\[
L_{i_1}\cap L_{i_2}\cap L_{i_3}=\varnothing
\]
whenever $i_1,i_2,i_3$ are distinct.

Thus the collection
$\{L_1,\ldots,L_r\}$
satisfies the hypotheses of the multinerve theorem stated above.

The associated multinerve has one vertex corresponding to each subcomplex $L_i$. Moreover, each connected component of
$L_i\cap L_j$
contributes one edge of the multinerve. Consequently, the multinerve is precisely the multigraph having
$r$
vertices and
$n_in_j$
parallel edges joining the vertices corresponding to $V_i$ and $V_j$.

Since this multigraph is connected, it is homotopy equivalent to a wedge of
\[
\left(\sum_{1\le i<j\le r}n_in_j\right)-r+1
\]
circles, completing the proof.
\end{proof}

\begin{remark}
    Recall that the $r$-independence complex of a graph $G$ is a simplicial complex whose faces correspond to those subsets $W\subset V(G)$ such that each connected component of the induced subgraph $G[W]$ has at most $r$ vertices. 
    One can verify that $\PF_3(G)$ is also the $2$-independence complex of $G$. 

    It is now known that for a complete multipartite graph $G$, the $r$-independence complex has the homotopy type of a wedge of $(r-1)$-spheres \cite[Theorem 3.5]{DS21}. 
    The exact formula for the number of spheres in $\PF_3(G)$ matches with our calculation. 
\end{remark}

The two cases we dealt with in this Section and the main result of the previous Section point towards the following conjecture and question. 

\begin{conjecture}
Let $G=K_{n_1,\ldots,n_r}$ be a complete multipartite graph and let $t\ge 2$.
Then the path-free complex $\PF_t(G)$ is homotopy equivalent to a wedge
of spheres. 
\end{conjecture}

\begin{question}
For $t\geq n_{m-1}-1$, determine the number of spheres appearing in the wedge decomposition of $\PF_t(G)$. 
\end{question}

\bibliographystyle{plain}
\bibliography{references} 

@article{DT,
author = {Taylan, Demet},
year = {2016},
month = {11},
pages = {},
title = {Matching Trees for Simplicial Complexes and Homotopy Type of Devoid Complexes of Graphs},
volume = {33},
journal = {Order},
doi = {10.1007/s11083-015-9379-3}
}

@article{SM,
author = {Saeedi Madani, Sara and Kiani, Dariush and Terai, Naoki},
year = {2011},
month = {01},
pages = {353–363},
title = {Sequentially Cohen-Macaulay path ideals of cycles},
volume = {54},
journal = {Bulletin mathématiques de la Société des sciences mathématiques de Roumanie}
}

@article{RW,
 author = {Russ Woodroofe},
 journal = {Proceedings of the American Mathematical Society},
 number = {10},
 pages = {3235--3246},
 publisher = {American Mathematical Society},
 title = {VERTEX DECOMPOSABLE GRAPHS AND OBSTRUCTIONS TO SHELLABILITY},
 urldate = {2024-12-14},
 volume = {137},
 year = {2009}
}

@article{AKK,
author = {Das, Kanoy and Roy, Amit and Saha, Kamalesh},
year = {2024},
month = {05},
pages = {},
journal = {arXiv:2405.15897},
title = {On the path ideals of chordal graphs},
doi = {10.48550/arXiv.2405.15897}
}

@article{Bi,
author = {Bijender},
year = {2024},
month = {05},
pages = {},
title = {Vertex Decomposability of the Stanley–Reisner Complex of a Path Ideal},
volume = {47},
journal = {Bulletin of the Malaysian Mathematical Sciences Society},
doi = {10.1007/s40840-024-01699-z}
}

@article{RG,
author = {Ehrenborg, Richard and Hetyei, Gabor},
year = {2006},
month = {08},
pages = {906-923},
title = {The topology of the independence complex},
volume = {27},
journal = {Eur. J. Comb.},
doi = {10.1016/j.ejc.2005.04.010}
}

@article{AH,
author = {He, Jing and Tuyl, Adam},
year = {2009},
month = {02},
pages = {},
title = {Algebraic Properties of the Path Ideal of a Tree},
volume = {38},
journal = {Communications in Algebra},
doi = {10.1080/00927870902998166}
}

@article{Kak,
title = {Independence complexes of chordal graphs},
journal = {Discrete Mathematics},
volume = {310},
number = {15},
pages = {2204-2211},
year = {2010},
issn = {0012-365X},
doi = {https://doi.org/10.1016/j.disc.2010.04.021},
author = {Kazuhiro Kawamura}
}

@article{FAVT,
author = {Francisco, Christopher and Tuyl, Adam},
year = {2005},
month = {12},
pages = {},
title = {Sequentially Cohen-Macaulay Edge Ideals},
volume = {135},
journal = {Proceedings of the American Mathematical Society},
doi = {10.2307/20534833}
}

@book{kozlov2007,
  title={Combinatorial Algebraic Topology},
  author={Kozlov, D.},
  isbn={9783540719625},
  lccn={2007933072},
  series={Algorithms and Computation in Mathematics},
  url={https://books.google.co.in/books?id=BfBHAAAAQBAJ},
  year={2007},
  publisher={Springer Berlin Heidelberg}
}

@book{JJ08,
author = {Jonsson, Jakob},
year = {2008},
month = {01},
title = {Simplicial Complexes of Graphs},
volume = {1928},
isbn = {978-3-540-75858-7},
publisher={Sprinter},
journal = {Lecture Notes in Mathematics},
doi = {10.1007/978-3-540-75859-4}
}

@article{RW09,
author = {Woodroofe, Russ},
year = {2009},
month = {11},
title = {Chordal and Sequentially {C}ohen-{M}acaulay Clutters},
volume = {18},
journal = {The electronic journal of combinatorics},
doi = {10.37236/695}
}

@article{ConcaDeNegri1999,
  author  = {Conca, Aldo and De Negri, Emanuela},
  title   = {{$M$}-sequences, graph ideals and ladder ideals of linear type},
  journal = {Journal of Algebra},
  volume  = {211},
  number  = {2},
  pages   = {599--624},
  year    = {1999}
}

@article{VanTuylVillarreal2008,
  author  = {Van Tuyl, Adam and Villarreal, Rafael H.},
  title   = {Shellable graphs and sequentially {C}ohen--{M}acaulay bipartite graphs},
  journal = {Journal of Combinatorial Theory, Series A},
  volume  = {115},
  number  = {5},
  pages   = {799--814},
  year    = {2008}
}

@book{BottTu1982,
  author    = {Bott, Raoul and Tu, Loring W.},
  title     = {Differential Forms in Algebraic Topology},
  series    = {Graduate Texts in Mathematics},
  volume    = {82},
  publisher = {Springer-Verlag},
  year      = {1982}
}

@article {ramras,
    AUTHOR = {Ramras, Daniel A.},
     TITLE = {Variations on the nerve theorem},
   JOURNAL = {Discrete Comput. Geom.},
  FJOURNAL = {Discrete \& Computational Geometry. An International Journal
              of Mathematics and Computer Science},
    VOLUME = {75},
      YEAR = {2026},
    NUMBER = {3},
     PAGES = {871--903},
      ISSN = {0179-5376,1432-0444},
   MRCLASS = {55U10 (55P10 55Q99)},
  MRNUMBER = {5009944},
       DOI = {10.1007/s00454-025-00809-3},
       URL = {https://doi.org/10.1007/s00454-025-00809-3},
}

@article {multinerve02,
    AUTHOR = {Fern\'andez, Ximena and Minian, El\'ias Gabriel},
     TITLE = {The cylinder of a relation and generalized versions of the
              nerve theorem},
   JOURNAL = {Discrete Comput. Geom.},
  FJOURNAL = {Discrete \& Computational Geometry. An International Journal
              of Mathematics and Computer Science},
    VOLUME = {63},
      YEAR = {2020},
    NUMBER = {3},
     PAGES = {549--559},
      ISSN = {0179-5376,1432-0444},
   MRCLASS = {55U10 (06A07 18B35 55P10 57Q10)},
  MRNUMBER = {4074333},
MRREVIEWER = {Jimmie\ D.\ Lawson},
       DOI = {10.1007/s00454-018-0028-7},
       URL = {https://doi.org/10.1007/s00454-018-0028-7},
}

@article {multinerve01,
    AUTHOR = {Colin de Verdi\`ere, \'Eric and Ginot, Gr\'egory and Goaoc,
              Xavier},
     TITLE = {Helly numbers of acyclic families},
   JOURNAL = {Adv. Math.},
  FJOURNAL = {Advances in Mathematics},
    VOLUME = {253},
      YEAR = {2014},
     PAGES = {163--193},
      ISSN = {0001-8708,1090-2082},
   MRCLASS = {54H05 (52A35)},
  MRNUMBER = {3148550},
MRREVIEWER = {Mircea\ Balaj},
       DOI = {10.1016/j.aim.2013.11.004},
       URL = {https://doi.org/10.1016/j.aim.2013.11.004},
}

@article{DS21,
 author = {Deshpande, Priyavrat and Singh, Anurag},
 title = {Higher independence complexes of graphs and their homotopy types},
 fjournal = {Journal of the Ramanujan Mathematical Society},
 journal = {J. Ramanujan Math. Soc.},
 issn = {0970-1249},
 volume = {36},
 number = {1},
 pages = {53--71},
 year = {2021},
 language = {English},
 url = {www.mathjournals.org/jrms/2021-036-001/2021-036-001-007.html},
 zbMATH = {7376455},
 Zbl = {1469.05168}
}

@article{TMAB,
  title={Topological methods},
  author={Bj{\"o}rner, Anders},
  journal={Handbook of combinatorics},
  volume={2},
  pages={1819--1872},
  year={1995},
  publisher={Amsterdam}
}

\end{document}